\documentclass[a4paper, 12pt]{amsart}
\usepackage[T1]{fontenc}
\usepackage{amssymb,amsthm,amsmath}
\usepackage{xspace}
\usepackage{enumerate}
\usepackage{color}
\usepackage{mathtools}

\newcommand{\R}{\mathbb{R}}

\newcommand{\N}{\mathbb{N}}

\newcommand{\Y}{\mathbb{Y}}
\newcommand{\G}{\mathbb{G}}

\newcommand{\norm}[1]{\left\lVert#1\right\rVert}

\newtheorem{theorem}{Theorem}[section]
\newtheorem{lemma}[theorem]{Lemma}
\newtheorem{proposition}[theorem]{Proposition}
\newtheorem{corollary}[theorem]{Corollary}

\theoremstyle{remark}
\newtheorem{remark}[theorem]{Remark}

\begin{document}

\title[]
{Fluctuation limits for mean-field interacting nonlinear Hawkes processes}
\author{Sophie Heesen, Wilhelm Stannat}
\address{
Institut f\"ur Mathematik\\
Technische Universit\"at Berlin \\
Stra{\ss}e des 17. Juni 136\\
D-10623 Berlin\\
Institut f\"ur Mathematik\\
Technische Universit\"at Berlin \\
Stra{\ss}e des 17. Juni 136\\
D-10623 Berlin\\
and\\
Bernstein Center for Computational Neuroscience\\
Philippstr. 13\\
D-10115 Berlin\\
Germany}
\email{heesen.sophie@googlemail.com, stannat@math.tu-berlin.de}
\date{Berlin, May 7, 2021}

\begin{abstract} 
We investigate the asymptotic behaviour of networks of interacting non-linear Hawkes processes 
modeling a homogeneous population of neurons in the large population limit. In particular, we prove a 
functional central limit theorem for the mean spike-activity thereby characterizing the asymptotic fluctuations 
in terms of a stochastic Volterra integral equation. Our approach differs from previous approaches in making 
use of the associated resolvent in order to represent the fluctutations as Skorokhod continuous mappings of 
weakly converging martingales. Since the Lipschitz properties of the resolvent are explicit, our analysis in 
principle also allows to derive approximation errors in terms of driving martingales. We also discuss 
extensions of our results to multi-class systems. 
\end{abstract}


\maketitle

\section{Introduction}

\noindent
The main purpose of this paper is to give a short proof of a functional central limit theorem for the 
fluctuations of a network of interacting Hawkes processes in the large population limit. Interacting 
Hawkes processes have been intensively studied in recent years as mathematical models for the statistical 
analysis of biological neural networks (cf. \cite{GalLoe, Chevallier1, Chevallier2, DelattreEtAl, 
DitLoe} and references therein). Results on their asymptotic behavior in the large population limit have 
been obtained in particular in the latter four references. Various functional laws of large numbers and 
(functional) central limit theorems (CLT henceforth) have been obtained therein in the respective 
settings. In particular, a CLT by simultaneously 
taking both, the number of neurons (or size of the network) and time to infinity has been obtained in 
\cite{DitLoe} and a functional CLT for age-dependent Hawkes processes in 
\cite{Chevallier2}. In contrast to the usual proof of functional CLTs via tightness properties, we will 
represent the fluctuation process as the unique solution of a stochastic Volterra integral equation 
driven by c\`adla\`g martingales that are shown to converge to a rescaled Brownian motion via the 
martingale CLT. This allows in addition to the corresponding result in \cite{Chevallier2} 
for an explicit representation of the limiting fluctuation process (see Remark \ref{RemarkCLT2} below 
for a precise comparison).   

Let us introduce our precise setting. Consider an $\mathbb{R}_{+}$-indexed filtered probability space 
$\left(\Omega,\mathcal{F},\mathbb{F},\mathbb{P}\right)$ and let $N_{i}$, $i\in\mathbb{N}$, be a sequence 
of iid $\mathbb{F}$-Poisson random measures with intensity measure $dzds$ on $\mathbb{R}_{+}
\times\mathbb{R}_{+}$. Let $f$ and $h$, modelling spike rate resp. synaptic weights, be 
two real-valued functions satisfying the following set of assumptions (A):  
\begin{enumerate}
    \item[(A.1)] $f\in C^{2}\left(\mathbb{R},\mathbb{R}_{+}\right)$, $f$ and $f'$ are Lipschitz 
                 continuous,                 
    \item[(A.2)] $h\in C^{1}\left(\mathbb{R}_{+},\mathbb{R}\right)$, $h(0)=0$
\end{enumerate}
Then, for each $N\in\mathbb{N}$, there exists a pathwise unique solution $Z^{N}=\left(Z^{N}_{1},\ldots , 
Z^{N}_{N}\right)$ of the integral equation 
\begin{align}\label{eq:HP}
     Z^{N}_{i}(t)&=\int\limits_{0}^{t}\int\limits_{0}^{\infty}1_{\left\{z\leq\lambda^{N}(s)\right\}}
     \;N_{i}\left(dz\times ds\right)\;,&& t\geq0\;,\;1\leq i\leq N\;,\\
    \lambda^{N}(t)&=f\left(\frac{1}{N}\sum\limits_{i=1}^{N}\int\limits_{0}^{t-}h(t-s)\;dZ^{N}_{i}(s)
    \right)\;,&& t>0\;,\label{eq:Def-small_Lambda}
\end{align}
(cf. \cite[Theorem 6]{DelattreEtAl}). Moreover, $Z^{N}$ is an interacting non-linear Hawkes process with 
parameters $\left(f,h,N\right)$ in the sense of \cite[Proposition 3(a), Definition 1]{DelattreEtAl}). In 
particular, $Z^{N}=\left(Z^{N}_{1},\ldots,Z^{N}_{N}\right)$ is a family of $\mathbb{F}$-counting 
processes satisfying 
\begin{equation} 
\label{eq:no-sim-jump}
    \mathbb{P}\left(\Delta Z^{N}_{i}(t)=1\;,\;\Delta Z^{N}_{j}(t)=1\right)=0 \;,  
    \quad i\ne j\;,\;t\geq0\;,
\end{equation}
with compensator (or cumulative intensity process) 
\begin{equation} 
\label{eq:Def-Lambda}
    \Lambda^{N}(t)=\int\limits_{0}^{t}\lambda^{N}(s)\;ds, \quad t\ge 0, 
\end{equation}
for each $1\leq i\leq N$. 

Delattre et. al. prove in \cite[Theorem 8]{DelattreEtAl} the following 
propagation of chaos result: for any $N\in\mathbb{N}$, there exists a family of iid Poisson processes 
$\bar{Z}_{i}$, $1\leq i\leq N$, with compensator $m$ being the unique solution of  
\begin{equation} 
\label{eq:Def-m}
    m(t)=\int\limits_{0}^{t}f\left(\int\limits_{0}^{s}h(s-u)\;dm(u)\right)\;ds\;,\quad t\ge 0, 
\end{equation}
such that for all $T > 0$ there exists some constant $C(T)$ 
\begin{equation} 
\label{eq:L1-Cvg-HP}
\mathbb{E}\left\lbrack\norm{Z^{N}_{i}-\bar{Z}_{i}}_{T}\right\rbrack 
\le\frac{C(T)}{\sqrt{N}}, \quad 1\le i\leq N\; ,N\in\mathbb{N}. 
\end{equation} 
Here, $\|y\|_T := \sup_{t\in[0,T]} |f(t)|$ denotes the supremum norm of a function $y : [0,T]\to\R$.  
This particularly yields a first order approximation of the Hawkes process $Z^{N}$ with the help of 
$N$ iid Poisson processes, each with intensity 
\begin{equation} 
\label{eq:Def-gamma} 
\lambda(t)=\frac{dm}{dt} (t) =f\left(x_t\right)\;,\quad t\ge 0\;.
\end{equation} 
with 
\begin{equation} 
\label{eq:Def-small_x}
x_{t}=\int\limits_{0}^{t}h(t-s)\;dm(s)\;,\quad t\ge 0\;.
\end{equation}
As an immediate consequence (cf. Proposition \ref{LLN}) the following 
functional law of large numbers for interacting non-linear Hawkes processes holds: 
\begin{equation} 
\label{eq:LLN}
\norm{\frac{1}{N}\sum\limits_{i=1}^{N}\left(Z^{N}_{i}-m\right)}_{T}\xrightarrow[N\to\infty]{L^{1}}0\;,
\quad T>0\;.
\end{equation}

\medskip 
In this paper we are now interested in the corresponding asymptotic fluctuations 
\begin{equation} 
\label{eq:Def-Y}
Y^{N}(t)=\frac{1}{\sqrt{N}} \sum\limits_{i=1}^{N} 
\left(Z^{N}_{i}(t)-m(t)\right)\;, \quad t\ge 0\;,\;N\in\mathbb{N}\; .
\end{equation}
The limiting behaviour of \eqref{eq:Def-Y} is linked to the asymptotic 
behaviour of the associated intensity processes $\lambda^N$ that can be 
written as 
\begin{equation}
\lambda^{N}(t)=f\left(x_{t}+\frac{X^{N}(t)}{\sqrt{N}}\right)\;,\quad t\ge 0\;,\;N\in\mathbb{N}\;,
\end{equation}
with 
\begin{equation} 
\label{eq:Def-X}
X^{N}(t)=\int\limits_{0}^{t}h(t-s)\;dY^{N}(s) = 
\int_0^t h' (t-s) Y^N (s)\, ds \; ,  
\end{equation} 
where we used integration by parts and $h(0) = Y^N (0) = 0$. 
A Taylor expansion of $f$ at $x_t$ yields the linear approximation 
\begin{equation} 
\label{eq:intensity_fluc}
\sqrt{N}\left(\lambda^{N}(t)-\lambda(t)\right) 
\approx f' \left( x_t\right) X^{N}(t) 
= f' \left( x_t\right)\int_0^t h'(t-s) Y^N (s) \, ds ,   
\end{equation}
which is the key to our analysis. Our main results are contained in Section \ref{CLTs}. We will first 
prove a weak convergence result for $Y^N$ and identify its limit in Theorem \ref{CLT1}, subsequently 
prove the weak convergence of $X^N$ in Corollary \ref{CLT2} and finally the weak convergence of $\sqrt{N}
\left(\lambda^{N}(t)-\lambda(t)\right)$ in Theorem \ref{CLT3}. We also discuss in Subsection \ref{ExtensionsSystems} extensions of these convergence results to multi-class systems. Section \ref{Proof1} contains some 
preliminary results used in the analysis of Section \ref{CLTs}. Finally, in Section \ref{Conclusion}, we apply our results to a second order approximation 
of the Hawkes process $Z^{N}$ by a system of $N$ identically distributed Cox processes.

\section{Main Results}
\label{CLTs}

\noindent 
Let us first introduce the Skorokhod metric. To this end denote by $D\left(I,\mathbb{R}\right)$, where  
$I\subseteq\mathbb{R}_{+}$, the space of all c\`{a}dl\`{a}g functions mapping $I\rightarrow\mathbb{R}$. 
For any $T>0$ let $\Gamma_{T}$ denote the set of all strictly increasing and continuous bijections on 
$\left\lbrack0,T\right\rbrack$ and for $x,y\in D\left(\left\lbrack0,T\right\rbrack,\mathbb{R}\right)$ put 
\begin{equation}
    d_{S}^{T}(x,y)=\inf\limits_{\gamma\in\Gamma_{T}}\max\left\{\norm{\gamma-id}_{T},\;
    \norm{x-y\circ\gamma}_{T}\right\},
\end{equation}
where $\norm{\cdot}_{T}$ denotes the supremum-norm on $\left\lbrack0,T\right\rbrack$. It is well-known 
that $d^{T}_{S}$ is a metric on $D\left(\left\lbrack0,T\right\rbrack,\mathbb{R}\right)$ and that on 
$D\left(\mathbb{R}_{+},\mathbb{R}\right)$ there exists a metric $d_{S}$ such that 
$\left(D\left(\mathbb{R}_{+},\mathbb{R}\right),\;d_{S}\right)$ is a separable and complete metric 
space and convergence in $d_{S}$ is equivalent to convergence in $d_{S}^{T}$ for all $T>0$ 
(see Chapter 3.5 of \cite{EK}). All following convergence results refer to the weak convergence topology 
associated with the Skorokhod topology on $D\left(\mathbb{R}_{+},\mathbb{R}\right)$ induced by $d_{S}$. 

\medskip 
A crucial ingredient in our analysis will be the following inhomogeneous Volterra integral equation of the 
second kind 
\begin{equation}
\label{eq:Volterra_2nd_kind} 
G (t) = \int_0^t \kappa (t,s) G (s) \, ds + F (t) \, , 0\le s\le t, 
\end{equation} 
with integral kernel 
\begin{equation}
\label{eq:Def-kappa}
    \kappa(t,s)=\int\limits_{s}^{t}f'\left(x_{u}\right)h'(u-s)\;du\;, \qquad
    s\in\left\lbrack0,t\right\rbrack\;,\;t\geq0\;,
\end{equation}
where $x_{t}$ is as in \eqref{eq:Def-small_x}. Assumption (A) implies the upper bound 
\begin{equation} 
\label{eq:upper_bound_kappa} 
|\kappa (t,s) |\le \|f^\prime \|_\infty \|h^\prime \|_{L^1[0,t]}   
\le \|f^\prime \|_\infty \|h^\prime \|_{L^1 [0,T]} =: M_T 
\end{equation} 
for $0\le s\le t\le T$. It is then well-known, that for local integrable $F$, the convolution  
\begin{equation} 
\label{eq:Def-Phi} 
\Phi (F ) (t) := \int_0^t K(t,s) F (s)\, ds + F(t) 
\end{equation} 
is the unique solution of \eqref{eq:Volterra_2nd_kind} (see Lemma \ref{Sol-asso-SIE}). Here, 
\begin{equation} 
\label{eq:Def-resolvent}
K(t,s) := \sum_{n=1}^\infty \kappa^{\otimes n}(t,s) \, , 0\le s \le t 
\end{equation} 
denotes the resolvent kernel associated with $\kappa$. The kernels $\kappa^{\otimes n}$ are iteratively 
defined by $\kappa^{\otimes1}\equiv\kappa$ and 
\begin{equation} 
\label{eq:Def-kappa_otimes}
    \kappa^{\otimes n+1}(t,s)=\int\limits_{s}^{t}\kappa(t,u)\kappa^{\otimes n}(u,s)\;du\;,  
     0\le s\le t\;,\;n\in\mathbb{N}\; .
\end{equation} 
see \cite{M}). It is well-known that 
$$ 
|\kappa^{\otimes n} (t,s)| \le \frac{M_t^n (t-s)^{n-1}} {(n-1)!}\, , n \in\mathbb N 
$$ 
(see, e.g. \cite{M}, page 16). It follows that the resolvent kernel converges locally uniformly and 
satisfies the upper bound 
$$ 
| K (t,s) |\le M_T e^{M_T} \, , 0\le s\le t\le T\, . 
$$ 
We will prove in the Appendix, Proposition \ref{prop3}, that $K(t,s)$ is Lipschitz in $t$ with locally 
bounded Lipschitz constant.

\begin{theorem} 
\label{CLT1}  
Assume that the pair of parameters $(f,h)$ satisfies (A) and let $Y^{N}$, $N\in\mathbb{N}$, be as in \eqref{eq:Def-Y}. Then, 
\begin{equation}\label{eq:CLT-Cvg1}
    Y^{N}\xRightarrow{\hspace{.1cm}N\to\infty\hspace{.2cm}}G_{Y}\hspace{1cm}\text{on}\hspace{1cm}D\left(\mathbb{R}_{+},\mathbb{R}\right)\;,
\end{equation}
where $G_{Y}$ is an It\^{o} process given by 
\begin{equation}\label{eq:Def-G_Y}
    G_{Y}(t)=\int\limits_{0}^{t}K(t,s)W_{\lambda}(s)\;ds+W_{\lambda}(t)\;,\hspace{1cm}t\geq0\;,
\end{equation}
with $K$ being as in \eqref{eq:Def-resolvent} and 
\begin{equation}
\label{eq:Def-W_gamma}
    W_{\lambda}(t)=\int\limits_{0}^{t}\sqrt{\lambda(s)}\;dW(s)\;,\hspace{1cm}t\geq0\;,
\end{equation}
with $W$ being a standard Brownian motion and $\lambda$ being as in \eqref{eq:Def-gamma}. In particular, 
$G_{Y}$ is the unique solution of the stochastic Volterra integral equation 
\begin{equation}
\label{eq:Volterra_eq}
    G(t)=\int\limits_{0}^{t}\kappa(t,s)G(s)\;ds+W_{\lambda}(t)\;,\hspace{1cm}t\geq0\;.
\end{equation} 
\end{theorem}

\begin{proof}
First note that $Y^{N}$ admits the following Doob-Meyer decomposition 
\begin{equation}
\label{eq:Doob-Meyer}
    Y^{N}=\bar{M}^{N}+A^{N} 
\end{equation}
with
\begin{equation} 
\label{eq:Def-M_bar}
\bar{M}^{N}=\frac{1}{\sqrt{N}}\sum\limits_{i=1}^{N}\left(Z^{N}_{i}-\Lambda^{N}\right)
\end{equation}
and 
\begin{equation} 
\label{eq:Def-A}
    A^{N}=\sqrt{N}\left(\Lambda^{N}-m\right)
\end{equation}
where $\Lambda^{N}$ and $m$ are as in \eqref{eq:Def-Lambda} 
and \eqref{eq:Def-m} respectively. Note that $\bar{M}^N$ 
is $a.s.$ locally bounded in $t$, since the compensated 
Hawkes processes $Z_i^N - \Lambda^N$ are locally bounded, and that 
$A^N$ is continuous. The crucial step now is to represent $Y^N$ as the unique solution of the Volterra integral equation 
\begin{equation} 
\label{eq:Volterra_Eq_Y}  
Y^N (t) = \bar{M}^N (t) + r^N (t) + \int_0^t \kappa (t,s) Y^N (s)\, ds 
\end{equation} 
with the (continuous) remainder 
\begin{equation} 
\label{eq:Definition_r} 
r^N (t) := A^N (t) - \int_0^t f' (x_s) X^N (s)\, ds \, . 
\end{equation}  
Indeed, \eqref{eq:Volterra_Eq_Y} follows from the above definition 
of $r^N$ using Fubini's theorem, since by \eqref{eq:Def-X} 
$$ 
\int_0^t f' (x_s) \int_0^s h' (s-r) Y^N (r) \, dr \, ds  
= \int_0^t \kappa (t,r) Y^N (r)\, dr\, . 
$$ 
It follows from Lemma \ref{Sol-asso-SIE} that 
\begin{equation} 
\label{eq:Volterra_Eq_Y_repr}  
Y^N (t) = \Phi \left( \bar{M}^N  + r^N \right) \,  . 
\end{equation}

The martingale CLT implies that 
\begin{equation*} 
\bar{M}^{N}\xRightarrow{\hspace{.1cm}N\to\infty\hspace{.2cm}}W_{\lambda} 
\text{ on } D\left(\mathbb{R}_{+},\mathbb{R}\right)\;,
\end{equation*}
where 
\begin{equation*}
    W_{\lambda}(t)=\int\limits_{0}^{t}\sqrt{\lambda(s)}\;dW(s)\;\hspace{1cm}t\geq0\;,
\end{equation*}
for some standard Brownian motion $W$ (see Proposition 
\ref{Cvg_M_bar}). Since $r^{N}$, $N\in\mathbb{N}$, 
converges ucp to $0$ (Lemma \ref{Cvg_R}), it follows that  
\begin{equation*}
d_{S}^{T}\left(\bar{M}^N  + r^N,\bar{M}^{N}\right)\leq\norm{r^{N}}_{T}\xrightarrow[N\to\infty]
{\mathbb{P}}0\;,\quad T>0 . 
\end{equation*}
This implies that the Skorohod distance $d_S$ between $\bar{M} + r^N$ and
$\bar{M}^N$, as processes in $D \left( \R_+ , \R \right)$, goes to $0$ as $N\to\infty$.
A generalization of Slutsky's theorem  
(e.g. cf. \cite[Theorem 3.1]{Bill}) now yields that 
\begin{equation*}
\bar{M}^{N} + r^N \xRightarrow{\hspace{.1cm}N\to\infty\hspace{.2cm}}W_{\lambda}\hspace{1cm}\text{on}\hspace{1cm}D\left(\mathbb{R}_{+},\mathbb{R}\right) . 
\end{equation*}
Since $\Phi$ is Skorokhod continuous by Proposition \ref{Cont_Phi} it follows that 
\begin{equation*}
    \Phi\left(\bar{M}^{N} + r^N \right)\xRightarrow{\hspace{.1cm}N\to\infty\hspace{.2cm}}
    \Phi\left(W_{\lambda}\right)\hspace{1cm}\text{on}\hspace{1cm}D\left(\mathbb{R}_{+},\mathbb{R} 
    \right)\; .
\end{equation*} 
Obviously, $\Phi\left(W_{\lambda}\right) = G_{Y}$ as defined in \eqref{eq:Def-G_Y} which implies the 
assertion. 
\end{proof}

\begin{corollary} 
\label{CLT2}
Assume that the pair of parameters $(f,h)$ satisfies (A) and let $X^{N}$, $N\in\mathbb{N}$, be as in \eqref{eq:Def-X}. Then, 
\begin{equation} 
\label{eq:CLT-Cvg2}
    X^{N}\xRightarrow{\hspace{.1cm}N\to\infty\hspace{.2cm}}G_{X}\hspace{1cm}\text{on}\hspace{1cm}D\left(\mathbb{R}_{+},\mathbb{R}\right)\;,
\end{equation}
where
\begin{equation}\label{eq:Def-G_X}
    G_{X}(t)=\int\limits_{0}^{t}h'(t-s)G_{Y}(s)\;ds\;,\hspace{1cm}t\geq0\;,
\end{equation}
with $G_{Y}$ being as in \eqref{eq:Def-G_Y}. In particular, the limiting process $G_{X}$ solves the 
stochastic integral equation 
\begin{equation}
\label{eq:SDE-G_X}
    G(t)=\int\limits_{0}^{t}h(t-s)f'\left(x_{s}\right)G(s)\;ds+\int\limits_{0}^{t}h(t-s)\;dW_{\lambda}(s)\;,\quad t\geq0\;,
\end{equation}
with $W_{\lambda}$ being as in \eqref{eq:Def-W_gamma}.
\end{corollary}

\begin{remark}
\label{RemarkCLT2} 
Corollary \ref{CLT2} should be compared with Theorem 5.6 of 
\cite{Chevallier2} for age-dependent Hawkes processes. Note that  
\eqref{eq:SDE-G_X} is the equivalent to equation (5.10) in 
\cite{Chevallier2}. Indeed, in the case where the intensity $\Psi (s,x)
\equiv \Psi (x)$ becomes independent of the age $s$, equation (5.10) 
for $\Gamma_t$ reduces to the integral equation 
$$ 
\Gamma_t  
= \int_0^t h(t-z) \langle u_z , \frac{\partial\Psi}{\partial y} 
(\cdot , \bar{\gamma} (z))\rangle \Gamma_z\, dz  
+ \int_0^t h(t-z) d W_z (1)\, , 
$$ 
since $\langle \eta_z , 1\rangle = 0$ for all $z$. Here, 
$W_t (1)$, $t\ge 0$, is a standard Brownian motion with 
variance $\int_0^t \Psi (\bar{\gamma} (z))\, dz$. Since $\Psi$ 
coincides with $f$ and $\bar{\gamma} (t)$ with $x_t$ in our present 
setting, the assertion now follows. The assumptions on $f$ in the 
present paper coincide with the assumptions made on $\Psi$ in 
\cite{Chevallier2}, however we require more regularity on $h$ ($C^1$ 
instead of H\"older-continuity).    
\end{remark}

\begin{proof} (of Corollary \ref{CLT2})  
Define the linear integral operator
\begin{equation}\label{eq:Def-Psi}
\Psi:D\left(\mathbb{R}_{+},\mathbb{R}\right)\to D\left(\mathbb{R}_{+},\mathbb{R}\right) 
\;,\quad\Psi\left(F\right)(t)=\int\limits_{0}^{t}F(s)h'(t-s)\;ds
\end{equation}
so that we can write $X^{N}(t) = \Psi\left(Y^{N}\right)(t)$. 
Proposition \ref{Operator-Estimation}, applied to $g(t,s) = h' (t-s)$, 
yields that 
\begin{align*}
\norm{\Psi\left(F\right)-\Psi\left(G\right)\circ\gamma}_{T} 
& \leq C_T \big( 2\left(\norm{F}_{L^1 [0,T]} + \norm{G}_{L^1 [0,T]}\right) \norm{id-\gamma}_{T} \\
  & \hspace{3cm} + \norm{F-G}_{L^1 [0,T]} \big) 
\end{align*} 
for any $F,G\in D\left(\mathbb{R}_{+},\mathbb{R}\right)$, $T>0$ and $\gamma\in\Gamma_{T}$. 
This estimation particularly implies that $\Psi$ is Skorokhod continuous (see Corollary \ref{Cont_Phi}). 
Hence, Theorem 
\ref{CLT1}, together with the continuous mapping theorem, yields 
\begin{equation*} 
X^N = \Psi\left(Y^{N}\right)\xRightarrow{\hspace{.1cm}N\to\infty\hspace{.2cm}}\Psi\left(G_{Y}\right) 
 = G_X \text{ on }D\left(\mathbb{R}_{+},\mathbb{R}\right)\, . 
\end{equation*} 
It remains to show that $G_{X}$ solves \eqref{eq:SDE-G_X}. To this end note that  
\begin{equation*}
G_{X}(t)=\int\limits_{0}^{t}h'(t-s)G_{Y}(s)\;ds=\int\limits_{0}^{t}h(t-s)\;dG_{Y}(s) 
\;,\hspace{1cm}t\geq0\;,
\end{equation*}
this time using $h(0)= G_{Y}(0) = 0$, so that \eqref{eq:Volterra_eq} now implies that 
\begin{align*}
G_X (t)  
& = \int_0^t h(t-s) f'\left(x_{s}\right) G_X (s)\;ds + \int_0^t h(t-s) dW_{\lambda}(s) \; .
\end{align*} 
\end{proof}

\begin{theorem}
\label{CLT3}
Assume that the pair of parameters $(f,h)$ satisfies (A) and let the intensity fluctuation be as in \eqref{eq:intensity_fluc}. Then, 
\begin{equation}\label{eq:CLT-Cvg3}
\sqrt{N}\left(\lambda^{N}-\lambda\right)\xRightarrow{\hspace{.1cm}N\to\infty\hspace{.2cm}}\sigma\hspace{1cm}\text{on}\hspace{1cm}D\left(\mathbb{R}_{+},\mathbb{R}\right)\;,
\end{equation}
where 
\begin{equation}
\label{eq:Def-sigma}
    \sigma(t)=f'\left(x_{t}\right)G_{X}(t)\;,\hspace{1cm}t\geq0\;,
\end{equation}
with $x_{t}$ and $G_{X}$ being as in \eqref{eq:Def-gamma} and \eqref{eq:Def-G_X} respectively.
\end{theorem}

\begin{proof}
First note that the intensity fluctuation can be rewritten as 
\begin{equation*}
\begin{aligned} 
    \sqrt{N}\left(\lambda^{N} (t) -\lambda (t)\right) & = \sqrt{N} \left( f\left( x_t  
       + \frac{X^N (t)}{\sqrt{N}}\right)   - f(x_t) \right) \\ 
       & = \int_0^1 f' \left(  x_t  + s\frac{X^N (t)}{\sqrt{N}}\right) \, ds X^N (t)  \\
       & =: F^{N} (t) X^{N} (t) 
\end{aligned} 
\end{equation*}
with 
\begin{equation*}
    F^{N}(t) := \int\limits_{0}^{1}f'\left(x_{t}+s\,\frac{X^{N}(t)}{\sqrt{N}}\right)\;ds 
     \;,\quad t\ge 0\;,\;N\in\mathbb{N}\;.
\end{equation*}
By Lipschitz continuity of $f'$ it holds 
\begin{equation*}
    \left|F^{N}(t)-f'\left(x_{t}\right)\right|\leq\frac{\|f'\|_{Lip}}{2\sqrt{N}} 
    \left|X^{N}(t)\right|\;,\quad t\geq0\;,\;N\in\mathbb{N}\;, 
\end{equation*}
such that 
\begin{equation*}
    \norm{F^{N}X^{N}-f'\left(x_{\cdot}\right)X^{N}}_{T}
    \le \frac{\|f'\|_{Lip}}{2\sqrt{N}}\sup\limits_{t\in\lbrack0,T\rbrack}\left|X^{N}(t)\right|^{2} 
    \xrightarrow[N\to\infty]{L^{1}}0\;,\quad T>0\;,
\end{equation*}
by Proposition \ref{L2-bdd_X}, which particularly implies 
\begin{equation*}
    d^{T}_{S}\left(F^{N}X^{N},f'\left(x_{\cdot}\right)X^{N}\right) 
    \xrightarrow[N\to\infty]{\mathbb{P}}0\;,\quad T>0\;.
\end{equation*}
Since $f'(x_t)$ is continuous it follows that the linear operator 
\begin{equation}
\label{eq:Def-Xi}
\Xi:D\left(\mathbb{R}_{+},\mathbb{R}\right)\to D\left(\mathbb{R}_{+},\mathbb{R}\right)\;,\quad 
\Xi\left(F\right)(t)=f'\left(x_{t}\right)F(t)
\end{equation}
is Skorokhod continuous, which implies by the continuous mapping theorem that 
\begin{equation*}
F^{N}X^{N}\xRightarrow{\hspace{.1cm}N\to\infty\hspace{.2cm}} f' (x_\cdot) G_{X} 
= \sigma \text{ on } D\left(\mathbb{R}_{+},\mathbb{R}\right) \, . 
\end{equation*}
\end{proof}

\subsection{Extensions to multi-class systems} 
\label{ExtensionsSystems} 

The results can be easily generalized to multi-class systems of 
interacting nonlinear Hawkes processes. More specifically fix 
$K\in\N$ and suppose that for each $N$ we are given $K$ classes 
of Hawkes processes $\{Z^N_{k,i} : 1\le i\le N_k\}$, $1\le k\le K$, 
where $N_1 , \ldots , N_K$ are positive integers with 
$N = N_1 + \ldots + N_K$ such that for each  
$k=1 , \ldots , K$ the sequence $(N_k)$ satisfies
$\lim_{N\to\infty} \frac{N_k}N = p_k \in ]0,1]$. 
Suppose that $Z_{k,i}^N$ are solutions of the following 
coupled systems of integral equations 
$$ 
\begin{aligned} 
Z_{k,i}^N (t) & = \int_0^t \int_0^\infty 1_{\{z\le \lambda_{k}^N (s) \}} N_{k,i} (dz \times ds) , 
t\ge 0,  \\ 
\lambda_k^N (t) & = f_k\left( \sum_{l=1}^K \frac{1}{N_l} \sum_{i=1}^{N_l} \int_0^{t-} h_{k,l} (t-s) 
dZ_{l,i}^N (s)\right) 
\end{aligned} 
$$ 
where $N_{k,i}$, $k = 1 , \ldots , K$, $i= 1 , \ldots , N_k$, 
are independent Poisson random measures with intensity measure 
$dz\, ds$ on $\R_+\times \R_+$ and $(f_k)$ (resp. $(h_{k,l})$) 
satisfy Assumption (A.1) (resp. (A.2)). Let $m_k$, $k =1 ,  
\ldots , K$, be the unique solution of the following coupled 
system of integral equations 
$$ 
m_{k} (t) = \int_0^t f_k \left( \int_0^s \sum_{l=1}^K h_{k,l} 
(s-u) \, dm_l (u) \right)\, ds \, ,\  t\ge 0, 
$$ 
and let $\bar{Z}_{k,i}^N$ be the independent Poisson processes 
defined by 
$$ 
\begin{aligned} 
\bar{Z}_{k,i}^N (t) 
& = \int_0^t \int_0^\infty 1_{\{z\le \lambda_{k} (s) \}} 
N_{k,i} (dz \times ds)\, ,\  t\ge 0,  
\end{aligned} 
$$ 
where 
\begin{equation}
\label{def:Lambda}
\lambda_k (t) = \frac{dm_k}{dt}(t) .
\end{equation}
Ditlivsen et. al. have shown in the proof of 
\cite[Theorem 1]{DitLoe} that there exists a constant $C(T)$ 
such that for each $N$ 
\begin{equation*} 
\mathbb{E}\left\lbrack\norm{Z^{N}_{k,i}-\bar{Z}_{k,i}^N}_{T}\right
\rbrack 
\le C(T) \sum_{k=1}^K \frac 1{\sqrt{N_k}}\,  , 
\ 1\le i\leq N_k \, ,\ N\in\mathbb{N}, 
\end{equation*}
(see in particular equation (2.11) on page 1847 in \cite{DitLoe}). One can then show as an 
extension of Theorem \ref{CLT1} that the process $\Y^N =(Y_1^N , \ldots , Y_K^N )$,  
with 
$$ 
Y_k^N (t) = \frac{1}{\sqrt{N_k}} \sum_{i=1}^{N_k} (Z_{k,i} (t) - m_k (t)) \, , \ 1\le k\le K,  
$$ 
converges weakly as $N\to\infty$ to the $K$-dimensional It\^o-process $\G_\Y =(G_1 , \ldots , G_K)$,  
with 
$$ 
G_k (t) = \int_0^t \sum_{l=1}^K K_{k,l} (t,s) W_{\lambda_l} (s)\, ds + W_{\lambda_k} (t) 
$$ 
with a matrix-valued resolvent kernel $K_{k,l}$, $1\le k,l\le K$, given by  
\begin{equation*} 
K_{k,l} (t,s) := \sum_{n=1}^\infty \kappa^{\otimes n}_{k,l} (t,s) 
\, , \ 0\le s \le t,  
\end{equation*} 
and this time  
$$ 
\kappa_{k,l}^{\otimes 1}(t,s) = \sqrt{\frac{p_k}{p_l}} \int_s^t f_k^\prime 
\left( \sum_{l=1}^K \int_0^u h_{k,l} (u-v) \, dm_l (v)\right)
h^\prime_{k,l} (u-s)\, du 
$$
and 
\begin{equation*} 
\kappa^{\otimes n+1}_{k,l}(t,s)= \sum_{l_1=1}^K \int\limits_{s}^{t}\kappa^{\otimes 1} (t,u)_{k,l_1} 
\kappa^{\otimes n}_{l_1 , l}(u,s)\;du\;,  0\le s\le t\;,\;n\in\mathbb{N}\; .
\end{equation*} 
Moreover, 
$$ 
W_{\lambda_k} (t) = \int_0^t \sqrt{\lambda_k} (s) dW_k (s) , \ t\ge 0, 
$$ 
for independent 1D-standard Brownian motions $W_k$, $1\le k\le K$, 
and $\lambda_k$ defined by \eqref{def:Lambda}. 

The matrix-valued resolvent kernel satisfies the same continuity properties as in the scalar-valued case. 
The proof of the weak convergence of $\Y$ can be carried out as a straightforward generalization of the 
scalar-valued case. Mainly only one new aspect shows up in the multi-class generalization 
$$ 
Y^N_k (t) = \bar{M}_k^N (t) + r^N_k (t) + \sum_{l=1}^K \int_s^t \kappa^{\otimes 1}_{k,l} (t,s) Y_l^N (s)\, ds 
$$
of \eqref{eq:Volterra_Eq_Y} due to the additional convergence $\frac{N_k}{N} \to p_k$, which causes a 
minor additional effort in the control of the remainder $r_k^N$. 

Similarly, one can show as an extension of Theorem \ref{CLT3}, the following weak convergence 
$$ 
\sqrt{N_k} \left( \lambda_k^N - \lambda_k\right) 
\xRightarrow{\hspace{.1cm}N\to\infty\hspace{.2cm}}\sigma_k , \ 1\le k\le K,  
$$ 
where 
$$ 
\sigma_k (t) = f_k^\prime \left( \sum_{l=1}^K \int_0^t h_{k,l} (t-s) \, dm_l (s)\right) 
\sum_{l=1}^K \sqrt{\frac{p_k}{p_l}} \int_0^t h_{k,l}^\prime (t-s) G_l (s)\, ds. 
$$

\section{Proofs}
\label{Proof1}

Throughout the whole section we assume that the pair of parameters $(f,h)$ satisfies (A).

\begin{lemma} 
\label{Estimate_Lambda}
Let $\Lambda^N$ be as in \eqref{eq:Def-Lambda}. Then,
\begin{equation*}
\mathbb{E}\left\lbrack\Lambda^{N}(t)\right\rbrack 
\leq f(0)t\exp\left(\|f'\|_\infty\norm{h}_{t}t\right)\;, 
\hspace{.5cm}t\geq 0\;.
\end{equation*}
\end{lemma}

\begin{proof}
Fix $N\in\mathbb{N}$ and note that the Lipschitz continuity of $f$ yields
\begin{align*}
    \mathbb{E}\left\lbrack\lambda^{N}(t)\right\rbrack 
    & \le  f(0) + \|f'\|_\infty \int_0^t \left|h(t-s)\right|\,\mathbb{E}\left\lbrack 
      \lambda^{N}(s)\right\rbrack\;ds\\
    & \le f(0) + \|f'\|_\infty \norm{h}_{t} \int_0^t \mathbb{E}\left\lbrack\lambda^{N}(s)\right\rbrack\;ds
\end{align*}
for any $t\ge 0$. Therefore  
\begin{equation*}
    \mathbb{E}\left\lbrack\lambda^{N}(t)\right\rbrack\leq f(0)\exp\left( \|f'\|_\infty \norm{h}_{t}t\right)\; ,
\end{equation*}
by Gronwall's inequality. In particular, we have 
\begin{equation*}
\mathbb{E}\left\lbrack\Lambda^{N}(t)\right\rbrack 
= \int_0^t\mathbb{E}\left\lbrack\lambda^{N}(s)\right\rbrack\;ds 
  \le f(0)t\exp\left(\|f'\|_\infty\norm{h}_{t}t\right)\; .
\end{equation*}
\end{proof}

\begin{proposition} 
\label{L2-bdd_X} 
Let $X^{N}$ be as in \eqref{eq:Def-X} and $T > 0$. Then there exists a constant 
$C(T)$, such that 
\begin{equation}
\mathbb{E}\left\lbrack\sup\limits_{t\in\lbrack0,T\rbrack}\left|X^{N}(t)\right|^{2}\right\rbrack 
\le C(T) \; \quad T>0\;,\;N\in\mathbb{N}.
\end{equation}
\end{proposition}

\begin{proof}
Integration by parts and $h(0)=Y^{N}(0)=0$ imply that
\begin{equation*}
X^{N}(t)\overset{\eqref{eq:Def-X}}{=} \int\limits_{0}^{t}h(t-s)\;dY^{N}(s) 
=\int\limits_{0}^{t}Y^{N}(s)h'(t-s)\;ds \, . 
\end{equation*}
It follows from \eqref{eq:Doob-Meyer} that 
\begin{align*}
    \left|X^{N}(t)\right|&\leq\int\limits_{0}^{t}\left|Y^{N}(s)\right|\;\left|h'(t-s)\right|\;ds\\
    & \leq\int\limits_{0}^{t}\left|\bar{M}^{N}(s)\right|\;\left|h'(t-s)\right|\;ds  
    + \int\limits_{0}^{t}\left|A^{N}(s)\right|\;\left|h'(t-s)\right|\;ds\, . 
\end{align*} 
Now, 
\begin{align*}
\left| A^{N}(s) \right| 
& \overset{\hspace{.25cm}\eqref{eq:Def-A}\hspace{.2cm}}{=} 
\sqrt{N}\left|\Lambda^{N}(s)-m(s)\right|\overset{\eqref{eq:Def-Lambda} 
+\eqref{eq:Def-gamma}}{\leq}\int\limits_{0}^{s}\sqrt{N}\left|\lambda^{N}(r)-\lambda(r)\right|\;dr\\
& = \int\limits_{0}^{s}\left| \int_0^1 f'\left(x_{r}+u\,\frac{X^{N}(r)}
    {\sqrt{N}}\right) \; du X^N (r) \right|\;dr \\ 
& \leq \|f'\|_\infty \int\limits_{0}^{s}\left|X^{N}(r)\right|\;ds
\end{align*}
and therefore 
\begin{align*}
\left|X^{N}(t)\right|  
& \leq\int\limits_{0}^{t} \left( \left|\bar{M}^{N}(s)\right| + \|f'\|_\infty \int\limits_{0}^{s} 
  \left|X^{N}(u)\right|\;du \right) \left|h'(t-s)\right| \; ds \\
& \leq \left( \norm{\bar{M}^{N}}_{t} + \|f'\|_\infty \int_0^t \left|X^{N}(u)\right|\, du \right)  
  \int_0^t \left|h'(t-s)\right|\;ds \\ 
& \leq \left( \norm{\bar{M}^{N}}_{t} + \|f'\|_\infty \int_0^t \left|X^{N}(u)\right|\, du \right)  
  \norm{h'}_{L^1 [0,t]} 
\end{align*}
holds for any $t\geq0$ and $N\in\mathbb{N}$. By Gronwall's lemma for right-continuous functions 
(e.g. cf. \cite[Lemma 4]{Gronwall}) it follows that  
\begin{equation*}
    \left|X^{N}(t)\right|\leq C_{0}(t)\norm{\bar{M}^{N}}_{t}\;,\hspace{1cm}t\geq0\;,
\end{equation*}
with
\begin{equation*}
    C_{0}(t)=\norm{h'}_{L^1 [0,t]} \exp\left(\|f'\|_\infty \norm{h'}_{L^1 [0,t]} t\right) 
    \; .
\end{equation*}
We particularly obtain 
\begin{equation*}
    \mathbb{E}\left\lbrack\sup\limits_{t\in\lbrack0,T\rbrack}\left|X^{N}(t)\right|^{2}\right\rbrack 
    \leq C^{2}_{0}(T)\mathbb{E}\left\lbrack\norm{\bar{M}^{N}}_{T}^{2}\right\rbrack\;,\hspace{1cm}T>0.
\end{equation*}
Since $\bar{M}^{N}$ is a c\`adl\`ag martingale, Doob's inequality now implies that 
\begin{equation*}
    \mathbb{E}\left\lbrack\norm{\bar{M}^{N}}_{T}^{2}\right\rbrack 
    \le 4 \mathbb{E} \left( \left( \bar{M}^N (T)\right)^2 \right) 
    = 4 \mathbb{E}\left\lbrack\left\lbrack\bar{M}^{N}\right\rbrack_{T}\right\rbrack\; .
\end{equation*}
Now, 
\begin{equation*}
    \left\lbrack\bar{M}^{N}\right\rbrack_{T} =\left\lbrack\frac{1}{\sqrt{N}} 
    \sum\limits_{i=1}^{N}Z^{N}_{i}\right\rbrack_{T}\overset{a.s.}{=} 
    \frac{1}{N}\sum\limits_{i=1}^{N}Z^{N}_{i}(T)\;, 
    \quad T>0\;,\;N\in\mathbb{N}\;,
\end{equation*}
since each $Z^{N}_{i}$ is a counting process with continuous cumulative intensity $t\mapsto\Lambda^{N}(t)$ and $Z^{N}$ satisfies \eqref{eq:no-sim-jump}. Hence  
\begin{align*}
    \mathbb{E}\left\lbrack\left\lbrack\bar{M}^{N}\right\rbrack_{T}\right\rbrack=\frac{1}{N}\sum\limits_{i=1}^{N}\mathbb{E}\left\lbrack Z^{N}_{i}(T)\right\rbrack=\mathbb{E}\left\lbrack\Lambda^{N}(T)\right\rbrack\;, \quad t\ge 0\;,\;N\in\mathbb{N}\;,
\end{align*}
such that the previous Lemma \ref{Estimate_Lambda} now yields the assertion with 
\begin{equation*}
C(T) =  C^{2}_{0}(T) 4 f(0)T \exp\left(\|f'\|_\infty \norm{h}_{T} T\right)\;,\quad T>0. 
\end{equation*}
\end{proof}

\begin{lemma}
\label{Cvg_R}
Let 
\begin{equation}\label{eq:Def-small_r}
r^{N}(t)=A^{N}(t)-\int_0^t f' (x_s) X^{N}(s)\;ds\;,\hspace{1cm}t\geq0\;,\;N\in\mathbb{N}\;, 
\end{equation}
with $X^{N}$ and $A^{N}$ being as in \eqref{eq:Def-X} and \eqref{eq:Def-A} 
respectively. Then, 
\begin{equation}
    \norm{r^{N}}_{T}\xrightarrow[\hspace{.1cm}N\to\infty\hspace{.1cm}]{L^{1}}0\;,\hspace{1cm}T>0\;.
\end{equation}
In particular, $r^{N}$, $N\in\mathbb{N}$, converges ucp to 0.
\end{lemma}

\begin{proof} 
Fix $N\in\mathbb{N}$ and note that 
\begin{align*}
A^{N}(t) 
& \overset{\hspace{.25cm}\eqref{eq:Def-A}\hspace{.2cm}}{=}\sqrt{N}\left(\Lambda^{N}(t)-m(t)\right)\overset{\eqref{eq:Def-Lambda}+\eqref{eq:Def-gamma}}{=}\int\limits_{0}^{t}\sqrt{N}\left(\lambda^{N}(s)-\lambda(s)\right)\;ds \\
& \overset{\hspace{.25cm}\eqref{eq:intensity_fluc}\hspace{.2cm}}{=}\int\limits_{0}^{t} 
\int\limits_{0}^{1}f'\left(x_{s}+u\,\frac{X^{N}(s)}{\sqrt{N}}\right)\;du X^{N}(s) \;ds
\end{align*}
holds for any $t\ge 0$. Consequently, 
$$ 
\begin{aligned} 
r^{N}(t) 
& = \int_0^t \left(\int_0^1 f'\left(x_{s}+u\,\frac{X^{N}(s)}{\sqrt{N}}\right)-f'\left(x_{s}\right)\;du\right) X^{N}(s)\;ds 
\end{aligned}
$$
for any $t\geq0$ and any $N\in\mathbb{N}$ and therefore 
\begin{align*}
|r^{N} (t)| 
& \le \int_0^t \int_0^1 \left|f'\left(x_{s} + u\,\frac{X^N (s)}{\sqrt{N}}\right) 
   -f' \left( x_{s} \right)\right|\; du \left|X^N (s)\right| \; ds \\
& \le \int_0^t \int_0^1 \frac{\|f'\|_{Lip}}{\sqrt{N}} u\;du  
      \left|X^{N}(s)\right|^2\;ds \, . 
\end{align*}
It follows that 
\begin{equation*}
\mathbb{E}\left\lbrack\norm{r^{N}}_{T}\right\rbrack 
\le \frac{\|f'\|_{Lip}}{2\sqrt{N}}\int\limits_{0}^{T}\mathbb{E}\left\lbrack\left|X^{N}(t)\right|^{2}
\right\rbrack\;dt\xrightarrow{N\to\infty}0\;,\hspace{1cm}T>0\;,
\end{equation*}
by the previous Proposition \ref{L2-bdd_X}.
\end{proof}

\begin{proposition} 
\label{Cvg_M_bar}   
Let $\bar{M}^{N}$, $N\in\mathbb{N}$, be as in \eqref{eq:Def-M_bar}. Then,
\begin{equation}
    \bar{M}^{N}\xRightarrow{\hspace{.1cm}N\to\infty\hspace{.2cm}}W_{\lambda}\hspace{1cm}\text{on}\hspace{1cm}D\left(\mathbb{R}_{+},\mathbb{R}\right)\;,
\end{equation}
where
\begin{equation}
    W_{\lambda}(t)=\int\limits_{0}^{t}\sqrt{\lambda(s)}\;dW(s)\;,\hspace{1cm}t\geq0\;,
\end{equation}
with $W$ being a standard Brownian motion.
\end{proposition}
\begin{proof}
The generalized martingale CLT, e.g. see \cite[Chapter 7, Theorem 1.4]{EK}, yields the assertion if the martingale sequence $\bar{M}^{N}$, $N\in\mathbb{N}$, satisfies
\begin{equation*}
    \norm{\Delta\bar{M}^{N}}_{T}\xrightarrow[N\to\infty]{L^{1}}0\;, \quad T>0\;,
\end{equation*}
and 
\begin{equation*}
    \left\lbrack\bar{M}^{N}\right\rbrack_{t}\xrightarrow[N\to\infty]{\mathbb{P}}\int\limits_{0}^{t}
    \lambda(s)\;ds\;,\quad t\geq0\;.
\end{equation*}
In order to see that $\bar{M}^{N}$, $N\in\mathbb{N}$, indeed satisfies both of above conditions, firstly 
observe 
\begin{equation*}
    \left|\Delta\bar{M}^{N}(t)\right|=\frac{1}{\sqrt{N}}\sum\limits_{i=1}^{N}\Delta Z^{N}_{i}(t)
    \overset{\text{a.s.}}{\leq}\frac{1}{\sqrt{N}}\;,\quad t>0\;,\;N\in\mathbb{N}\;,
\end{equation*}
since each $Z^{N}_{i}$ is a counting process with continuous cumulative intensity $t\mapsto\Lambda^{N}(t)$ 
and $Z^{N}$ satisfies \eqref{eq:no-sim-jump}. Thus, it follows 
\begin{equation*}
    \mathbb{E}\left\lbrack\norm{\Delta\bar{M}^{N}}_{T}\right\rbrack\leq\frac{1}{\sqrt{N}}
    \xrightarrow{N\to\infty}0\;,\quad T>0\;.
\end{equation*}
As derived in the proof of Proposition \ref{L2-bdd_X} it holds:
\begin{equation*}
    \left\lbrack\bar{M}^{N}\right\rbrack_{t}\overset{\text{a.s.}}{=} 
    \frac{1}{N}\sum\limits_{i=1}^{N}Z^{N}_{i}(t)\;,\quad t\geq0\;,\;N\in\mathbb{N}\;,
\end{equation*}
such that \eqref{eq:LLN} yields 
\begin{align*}
    \left\lbrack\bar{M}^{N}\right\rbrack_{t}-m(t)\overset{\text{a.s.}}{=} 
    \frac{1}{N}\sum\limits_{i=1}^{N}\left(Z^{N}_{i}(t)-m(t)\right)\xrightarrow[N\to\infty]{L^{1}}0\; 
     \quad t\geq0\;, 
\end{align*}
for $m$ being as in \eqref{eq:Def-m}. Since $m(t) = \int_0^t \lambda (s)\, ds$ by \eqref{eq:Def-gamma} 
this implies the assertion.  
\end{proof}

\section{Conclusion and Outlook}\label{Conclusion}

\noindent 
Our results can be applied to the following construction of a second order approximation of $Z^{N}$: 
Consider any probability space $\left(\Omega,\mathcal{F},\mathbb{P}\right)$ on which there exist a 
standard Brownian motion $W$ and Poisson random measures $\pi_{i}$, $i\in\mathbb{N}$, with intensity 
measure $dzds$ on $\mathbb{R}_{+}\times\mathbb{R}_{+}$ such that $W$ and $\pi_{i}$, $i\in\mathbb{N}$, are 
all independent. Assume that the pair of parameters $(f,h)$ satisfies (A). 

We then solve \eqref{eq:Def-m} to obtain $m$ and its derivative $\lambda$. We can then construct the 
rescaled Brownian motion $W_{\lambda}$ from $W$ (see \eqref{eq:Def-W_gamma}) and subsequently, use 
$W_{\lambda}$ to construct
\begin{equation}
\sigma=\Xi\circ\Psi\circ\Phi\left(W_{\lambda}\right)\;,
\end{equation}
where the operators $\Phi$, $\Psi$ and $\Xi$ are defined in 
\eqref{eq:Def-Phi}, \eqref{eq:Def-Psi} and \eqref{eq:Def-Xi} 
respectively. Instead of using the resolvent kernel $K$ in 
\eqref{eq:Def-Phi}, one can also simply solve the stochastic 
integral equation \eqref{eq:SDE-G_X} for $G_X$ and subsequently 
obtain $\sigma (t) = f^\prime (x_t) G_X (t)$ according to 
\eqref{eq:Def-sigma}. 

For fixed $N$ we can then construct the asymptotic intensity 
\begin{equation}
\widehat{\lambda}^{N} = \lambda + \frac{\sigma}{\sqrt{N}}\;, 
\end{equation} 
and subsequently define 
\begin{equation}
\widehat{Z}^{N}_{i}(t)=\int\limits_{0}^{t}\pi_{i}\left(\left\lbrack0,\widehat{\lambda}^{N}(s) 
\right\rbrack \times ds\right)\;,\hspace{.5cm}t\ge 0,\;1\leq i\leq N\;.  
\end{equation}
Since $\widehat{\lambda}^{N}$ is independent of $\pi_{i}$, $i\in\mathbb{N}$, each $\widehat{Z}^{N}_{i}$ is 
an inhomogeneous Cox process (or doubly stochastic Poisson process) with stochastic intensity 
$\widehat{\lambda}^{N}$. It is easy to verify that there exists a sequence of iid Poisson processes 
$\Pi_{i}$, $i\in\mathbb{N}$, with intensity $\lambda$ such that 
\begin{equation} 
\label{eq:L1-Cvg-Approx}
\mathbb{E}\left\lbrack\norm{\widehat{Z}^{N}_{i}-\Pi_{i}}_{T}\right\rbrack 
\le\frac{T}{\sqrt{N}}\mathbb{E}\left\lbrack\norm{\sigma}_{T}\right\rbrack 
\xrightarrow{N\to\infty}0\;\hspace{1cm}T>0\;,\;i\in\mathbb{N}\;.
\end{equation}
Now, consider a Hawkes process $Z^{N}$ with parameters $(f,h,N)$ constructed as the pathwise unique 
solution of \eqref{eq:HP}, not necessarily sharing the same probability space as that of 
$\widehat{Z}^N$. Theorem \eqref{CLT3} implies that the distribution of its intensity process $\lambda^{N}$ 
can be approximated by $\widehat{\lambda}^{N}$. Moreover, in view of \eqref{eq:L1-Cvg-HP} and 
\eqref{eq:L1-Cvg-Approx}, $Z^{N}$ and $\widehat{Z}^{N}$ have the same behaviour in the large population 
limit. In conclusion, the system of Cox processes $\widehat{Z}^{N}$ approximates the behavior of the 
Hawkes process $Z^{N}$. However, a simulation study is necessary to evaluate the performance of 
$\widehat{Z}^{N}$, especially in comparison to an approximating system of simple Poisson processes. 
From a theoretical point of view approximating the behavior of the Hawkes process $Z^{N}$ by the system of 
Cox processes $\widehat{Z}^{N}$ should be superior to an approximation via a system of Poisson processes, 
because $\widehat{Z}^{N}$ maintains a random intensity process. This is especially desirable considering 
the Hawkes process $Z^{N}$ as a mathematical model for the statistical analysis of biological neural 
networks. 

\medskip 
\noindent 
\textbf{Numerical illustration} 

\noindent 
Consider as numerical illustration the 1st and 2nd order 
approximations $\lambda$ and $\hat{\lambda}^N$ of $\lambda^N$ in the 
case of a Hawkes network with $N= 50$ neurons, with synaptic kernel 
$h(t) = \beta^2 te^{- \beta t}$ for $\beta > 0$ and rate function 
$f(x) = \frac{1}{1 + e^{-\gamma (x-0.5)}}$ for $\gamma > 0$. 
Since the limiting intensity satisfies  
$$ 
\lambda (t) = f\left( \int_0^t h(t-s) \lambda (s)\, ds\right) 
$$ 
(cf. \eqref{eq:Def-gamma}), any asymptotically stable value  
$\lambda$ has to satisfy the equation $\lambda  = f \left( 
\int_0^\infty h(s) \lambda \, ds \right) = f(\lambda )$.
Depending on the number and stability properties of the fix 
points $f(x) = x$, we can now distinguish three different 
regimes for 
the dynamical behavior of the network. In the case $\gamma \le 4$ 
the only fix point is $x_* = 0.5$, which is globally asymptotically 
stable for $\gamma < 0.5$. Beyond the critical case 
$\gamma = 0.5$ $x_*$ becomes unstable and two other stable fix  
points $0 < x_- < 0.5 < x_+ < 1$ arise. 

\medskip 
\centerline{
\includegraphics[scale=0.4]{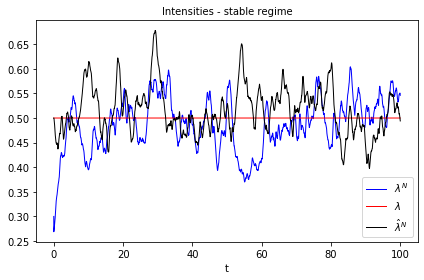}
\hspace{0.5cm}
\includegraphics[scale=0.4]{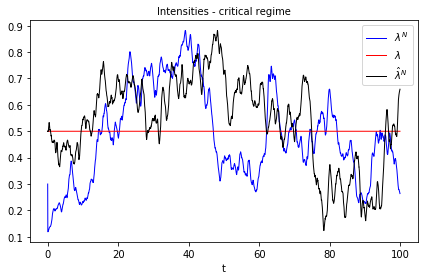}
}

\centerline{
\includegraphics[scale=0.4]{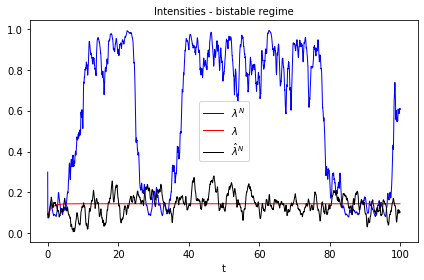}
\hspace{0.5cm}
\includegraphics[scale=0.4]{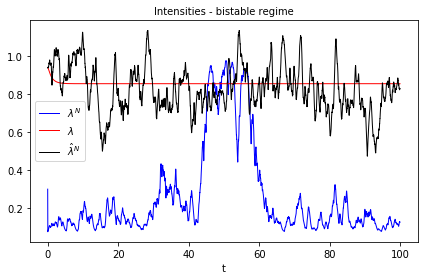}
}

\noindent 
The simulations indicate that fluctuations of $\hat{\lambda}^N$ and 
$\lambda^N$ are of the same order in the stable (with $\gamma = 2$,  
$\beta = 2$) and the critical regime (with $\gamma = 4$, 
$\beta = 2$), improving the trivial 1st order approximation given 
by $\lambda$. In the bistable regime (with $\gamma = 5$ and  
$\beta = 4$) however, a satisfactory approximation of $\lambda^N$ 
is only obtained around one of the stable fix points. 

\subsection*{Acknowlegdement} 
We thank two referees for several instructive comments. 
This work has been funded by Deutsche Forschungsgemeinschaft (DFG) through grant CRC 910 "Control of self-organizing nonlinear systems: Theoretical methods and concepts of application", Project (A10) "Control of stochastic mean-field equations with applications to brain networks". 

\bibliographystyle{amsalpha}
\bibliography{bibliography}

\appendix

\section{\, }\label{appendixVt}

\subsection{Law of Large Numbers}

\begin{proposition} 
\label{LLN}
Assume that the pair of parameters $(f,h)$ satisfies (A). Then,
\begin{equation*}
\norm{\frac{1}{N}\sum\limits_{i=1}^{N}\left(Z^{N}_{i}-m\right)}_{T}\xrightarrow[N\to\infty]{L^{1}}0\;,\hspace{1cm}T>0\;,
\end{equation*}
where $m$ is the unique solution of 
\begin{equation*}
    m(t)=\int\limits_{0}^{t}f\left(\int\limits_{0}^{s}h(s-u)\;dm(u)\right)\;ds\;,\hspace{1cm}t\geq0\;.
\end{equation*}
\end{proposition}
\begin{proof}
The assertion easily follows from \cite[Theorem 8]{DelattreEtAl}. Delattre et. al. particularly show that, 
for any $N\in\mathbb{N}$, there exists a sequence of iid Poisson processes $\bar{Z}_{i}$,  
$1\leq i\leq N$, with cumulative intensity $m$ such that 
\begin{equation*}
    \mathbb{E}\left\lbrack\norm{Z^{N}_{i}-\bar{Z}_{i}}_{T}\right\rbrack 
    \leq\frac{C(T)}{\sqrt{N}}\;,\quad T>0\;,\;1\leq i\leq N\;,\;N\in\mathbb{N}\;, 
\end{equation*}
for some constant $t\mapsto C(t)$. Hence, 
\begin{equation*}
    \mathbb{E}\left\lbrack\norm{\frac{1}{N}\sum\limits_{i=1}^{N}\left(Z^{N}_{i}-m\right)}_{T}
    \right\rbrack\leq\frac{C(T)}{\sqrt{N}}+\mathbb{E}\left\lbrack\norm{M^{N}}_{T}\right\rbrack\;, 
    \quad T>0\;,\;N\in\mathbb{N}\;,
\end{equation*}
by the triangle inequality, where 
\begin{equation*}
    M^{N}=\frac{1}{N}\sum\limits_{i=1}^{N}\left(\bar{Z}_{i}-m\right)\;,\hspace{1cm}N\in\mathbb{N}\;
\end{equation*}
is a martingale with 
\begin{equation*}
\mathbb{E}\left\lbrack\left( M^{N} (T)\right)^2 \right\rbrack 
= \text{Var}\left(\frac{1}{N}\sum\limits_{i=1}^{N}\bar{Z}_{i}(t)\right) 
= \frac{m(t)}{N}\;,\quad t\ge 0\;,\;N\in\mathbb{N}\;,
\end{equation*}
such that applying Jensen's inequality and Doob's inequality we obtain for $T>0$ and 
$N\in\mathbb{N}$ that 
\begin{equation*} 
\begin{aligned} 
\mathbb{E}\left\lbrack \norm{M^{N}}_{T}\right\rbrack^2  
& \le 4\mathbb{E}\left\lbrack\left( M^{N} (T)\right)^2 \right\rbrack  
\le 4 \frac{m(T)}{N} , 
\end{aligned} 
\end{equation*}
which implies the assertion. 
\end{proof}

\subsection{Skorokhod continuity of convolution integrals} 

\begin{proposition}
\label{Operator-Estimation}
Let $g:\mathbb{R}_{+}\times\mathbb{R}_{+}\to\mathbb{R}$ be measurable and suppose that for all $T > 0$ 
there exists some constant $C_T$ such that 
\begin{equation*}
\begin{aligned} 
& \hphantom{qi}\sup\limits_{0\le s\le t\le T} \left| g(t,s) \right| \le C_T \; , \\ 
& \sup\limits_{0\le s\le t_1 \le t_2\le T} \left|g(t_{1},s)-g(t_{2},s)\right| 
  \le C_T \left|t_{1}-t_{2}\right| \; . 
\end{aligned} 
\end{equation*}
Then the integral operator 
\begin{equation*}
    \Psi_{g}: D\left(\mathbb{R}_{+},\mathbb{R}\right)\to D\left(\mathbb{R}_{+},\mathbb{R}\right)\;,\hspace{1cm} 
    \Psi_{g}(F)(t) = \int\limits_{0}^{t} g(t,s)F(s)\;ds 
\end{equation*}
is well-defined and satisfies 
\begin{align*}
    \norm{\Psi_{g}(F)-\Psi_{g}(G)\circ\gamma}_{T} 
    & \le 2C_T \left(\norm{F}_{L^{1}\lbrack0,T\rbrack} + \norm{G}_{L^{1}\lbrack0,T\rbrack}\right)
          \norm{id-\gamma}_{T} \\ 
       & \hspace{1.5cm} + C_T \norm{F-G}_{L^1 [0,T]} \, . 
\end{align*}
\end{proposition}

\begin{proof}
Clearly, the integral $\Psi_g (F)(t)$ is well-defined for all $t$ and in fact continuous, hence in 
particular in $D\left(\mathbb{R}_{+},\mathbb{R}\right)$. Now, fix $T > 0$, $\gamma\in\Gamma_{T}$ and 
let $t\in\lbrack0,T\rbrack$. Then, for any $F,\,G\in D\left(\mathbb{R}_{+},\mathbb{R}\right)$ it holds 
\begin{align*} 
\left|\Psi_{g}(F)(t)\right.&-\left.\Psi_{g}(G)\circ\gamma(t)\right|
    =\left|\int\limits_{0}^{t}g(t,s)F(s)\;ds-\int\limits_{0}^{\gamma(t)}g(\gamma(t),s)F(s)\;ds\right|\\
    & \le \int\limits_{0}^{t\wedge\gamma(t)}\left|g(t,s)F(s)-g(\gamma(t),s)G(s)\right|\;ds\\
    & \hspace{2cm}+\int\limits_{t\wedge\gamma(t)}^{t}\left|g(t,s)F(s)\right|\;ds 
      +\int\limits_{t\wedge\gamma(t)}^{\gamma(t)}\left|g(\gamma(t),s)G(s)\right|\;ds\, . 
\end{align*}
We can now further bound all three terms from above as follows 
\begin{align*}
& \int_0^{t\wedge\gamma(t)}\left|g(t,s)F(s) 
  - g(\gamma(t),s)G(s)\right| ds \\
& \le\int\limits_{0}^{t\wedge\gamma(t)}\left|\left(g(t,s)-g(\gamma(t),s)\right)F(s)\right| ds 
+ \int\limits_{0}^{t\wedge\gamma(t)} \left|g(\gamma(t),s) 
 \left(F-G\right)(s)\right|ds \\
    & \le C_T \left( \norm{F}_{T}\norm{id-\gamma}_{T} + \norm{F-G}_{L^{1}\lbrack0,T\rbrack}\right) 
\end{align*}
and
\begin{align*}
    \int\limits_{t\wedge\gamma(t)}^{t}\left|g(t,s)F(s)\right|\;ds 
    & + \int\limits_{t\wedge\gamma(t)}^{\gamma(t)}\left|g(\gamma(t),s)G(s)\right|\;ds\nonumber\\
    & \hspace{2cm} \le C_T \left(\norm{F}_{T}+\norm{G}_{T}\right)\norm{id-\gamma}_{T}  
\end{align*}
which implies the assertion.  
\end{proof}

\begin{corollary}
\label{Cont_Phi} 
Let $g$ be as in Proposition \ref{Operator-Estimation}. Then the convolution operator 
\begin{equation*}
    \Phi_{g}: D\left(\mathbb{R}_{+},\mathbb{R}\right)\to D\left(\mathbb{R}_{+},\mathbb{R}\right) 
     \;,\quad 
    \Phi_{g}(F)(t) = \int\limits_{0}^{t} g(t,s)F(s)\;ds + F(t) 
\end{equation*}
is Skorokhod continuous.
\end{corollary} 

\begin{proof}
Take $F,F_{1},F_{2},\ldots\in D\left(\mathbb{R}_{+},\mathbb{R}\right)$ such that $F_{n}\to F$ as 
$n\to\infty$ in $d_{S}$. Fix $T>0$. Then there exists $\left(\gamma^{T}_{n}\right)_{n\in\mathbb{N}}
\subset\Gamma_{T}$ such that 
\begin{equation*}
\norm{id-\gamma^{T}_{n}}_{T}\xrightarrow{n\to\infty}0\;, \quad  
\norm{F_{n}-F\circ\gamma^{T}_{n}}_{T}\xrightarrow{n\to\infty}0\;.
\end{equation*}
The previous Proposition \ref{Operator-Estimation} now implies that 
\begin{align*}
&\norm{\Phi_g \left(F_{n}\right)-\Phi_g \left(F\right)\circ\gamma^{T}_{n}}_{T} 
\le 2C_T \left(\norm{F_{n}}_{L1 [0,T]} + \norm{F}_{L^1 [0,T]}\right)  
  \norm{id-\gamma^{T}_{n}}_T \\
& \hspace{4.5cm} +C_T \norm{F_{n}-F}_{L^1 \lbrack0,T\rbrack} + \norm{F_{n}-F\circ\gamma^{T}_{n}}_{T}
\rightarrow 0  
\end{align*}
as $n\to\infty$, since 
\begin{equation*}
\begin{aligned} 
&\norm{F_{n} - F}_{L^1 [0,T]} 
\le \norm{F_{n} - F\circ\gamma_n^T}_{L^1 [0,T]} + \norm{F\circ\gamma_n^T- F}_{L^1 [0,T]} \\ 
& \hspace{2cm} 
\le \norm{F_{n} - F\circ\gamma_n^T}_{T} T + \norm{F\circ\gamma_n^T- F}_{L^1 [0,T]} \to 0 
\, , n\to\infty\, . 
\end{aligned} 
\end{equation*}
Therefore $\lim_{n\to\infty} d^{T}_{S}\left(\Phi_g \left(F_{n}\right),\Phi_g \left(F\right)\right) = 0$, 
which implies the assertion, since $T$ was arbitrary.  
\end{proof}

\subsection{Properties of the resolvent kernel K} 
\label{Prop_K}

Recall the definition \eqref{eq:Def-resolvent} of the resolvent $K(t,s)$. We first note the following uniqueness result: 

\begin{lemma} 
\label{Sol-asso-SIE}
Let $F$ and $G$ be locally integrable with 
\begin{equation} 
\label{eq:Def-V}
    G (t) =  \int\limits_{0}^{t}\kappa(t,s)G(s)\;ds + F(t) \;,\quad t\ge 0 . 
\end{equation}
Then $G = \Phi \left(F\right)$. 
\end{lemma}

\begin{proof}
\eqref{eq:Def-V} is equivalent with  
\begin{equation*}
    \Phi\left(F\right)(t) - G(t) = \int\limits_{0}^{t}K(t,s) F(s)\; ds - \int\limits_{0}^{t}\kappa(t,s) G(s)\;ds . 
\end{equation*}
Now inserting \eqref{eq:Def-V} for $F(s)$ in the first integral und using Fubini's theorem now 
implies that 
\begin{equation*}
\begin{aligned} 
\Phi & (F) (t) - G(t) \\ 
& = \int_0^t K(t,s) ( G(s) - \int_0^s \kappa (s,u) G(u)\; du )\, ds  
 - \int_0^t\kappa (t,s) G(s)\; ds \\ 
& = \int_0^t ( K(t,s) - \kappa (t,s) ) G(s) \, ds  
 - \int_0^t \int_u^t K(t,s)\kappa (s,u) \; ds \; G(u)\; du  \\ 
& = 0 \, ,  
\end{aligned} 
\end{equation*}
since 
$\int_u^t K(t,s) \kappa (s,u) \, ds = K(t,u) - \kappa (t,u)$. 
\end{proof}

The following Lipschitz continuity of $K(t,s)$ w.r.t. $t$ is also needed.   

\begin{proposition}
\label{prop3}
Assume that the pair of parameters $(f,h)$ satisfies (A) and let $T > 0$. Then there exists a 
constant $C_T$ such that 
\begin{equation}
    \sup\limits_{s\in\lbrack0,t_{1}\wedge t_{2}\rbrack}\left|K(t_{1},s)-K(t_{2},s)\right| 
     \le C_T \left| t_{1}-t_{2} \right|\;,\quad 0\le t_1, t_2 \le T.
\end{equation} 
\end{proposition}

\begin{proof}
First note that for all $0\le s\le t$ the resolvent kernel satisfies the following integral equation  
\begin{align*}
K(t,s) = \kappa(t,s)+\sum\limits_{n=1}^{\infty} 
\kappa^{\otimes n+1}(t,s)  
=  \kappa(t,s)  + \int\limits_{s}^{t}\kappa(t,u)K(u,s)\;du . 
\end{align*}
We obtain for $s\le t_1\le t_2 \le T$ that 
\begin{equation} 
\label{eq:resolvent_lipschitz} 
\begin{aligned} 
& \left| K(t_{1},s) -K(t_{2},s)\right| 
 \le \left|\kappa(t_{1},s)-\kappa(t_{2},s)\right| \\ 
& \qquad\qquad 
+\left|\int_s^{t_{1}}\kappa(t_{1},u)K(u,s)\;du 
-\int_s^{t_{2}}\kappa(t_{2},u)K(u,s)\;du\right| \, .  
\end{aligned}
\end{equation} 
The integral kernel $\kappa (t,s)$ is Lipschitz continuous w.r.t. $t$ with 
Lipschitz-constant $M_{t}$, since 
\begin{equation} 
\label{eq:Lipschitz_kappa} 
\begin{aligned} 
&\left|\kappa(t_{1},s)-\kappa(t_{2},s)\right| 
= \left|\int\limits_{s}^{t_{1}}f'\left(x_{u}\right)h'(u-s)\;du-\int\limits_{s}^{t_{2}} 
   f'\left(x_{u}\right)h'(u-s)\;du\right| \\
& \hspace{2.8cm}\le\int\limits_{t_{1}}^{t_{2}}\left|f'\left(x_{u}\right)h'(u-s)\right|\;du 
   \le \|f\|_\infty \|h'\|_{L^1 [0,T]} \left| t_{2}-t_{1} \right|\, . 
\end{aligned} 
\end{equation} 
Using the estimate  
\begin{align*}
 & \left|\int\limits_{s}^{t_{1}}\kappa(t_{1},u)K(u,s)\;du-\int\limits_{s}^{t_{2}}\kappa(t_{2},u)K(u,s)\;du\right|\\
 & \hspace{1.3cm}\leq\int\limits_{s}^{t_{1}}\left|\left(\kappa(t_{1},u)-\kappa(t_{2},u)\right)K(u,s)\right|\;du+\int\limits_{t_{1}}^{t_{2}}\left|\kappa(t_{2},u)K(u,s)\right|\;du\\
 & \hspace{1.3cm}\le  M_T e^{M_T} \left( \int\limits_{s}^{t_{1}}\left|\kappa(t_{1},u)-\kappa(t_{2},u)\right|\;du+\int\limits_{t_{1}}^{t_{2}}\left|\kappa(t_{2},u)\right|\;du\right)\\
 & \hspace{1.3cm}\le M_T e^{M_T} \left(  \|f\|_\infty \|h'\|_{L^1 [0,T]}\left|t_{2}-t_{1}\right|  + 
     M_T \left|t_{2}-t_{1}\right|\right) ,  
\end{align*}
we now obtain the assertion.  
\end{proof}
\end{document}